\documentclass[12pt]{amsart}
\usepackage[utf8]{inputenc}
\usepackage{amssymb, eurosym}
\usepackage{amsmath}
\usepackage{amsthm}
\usepackage[left=2cm,right=2cm]{geometry}
\usepackage{amsfonts}
\usepackage[all]{xy} 

\theoremstyle{definition}
\newtheorem{defi}{Definition}[subsection]

\newtheorem{teo}[defi]{Theorem}
\newtheorem{cor}[defi]{Corollary}

\theoremstyle{definition}
\newtheorem{obs}[defi]{Remark}

\DeclareMathOperator{\rad}{\operatorname{rad}}
\DeclareMathOperator{\md}{\operatorname{mod}}
\DeclareMathOperator{\ind}{\operatorname{ind}}

\DeclareMathOperator{\Hom}{\operatorname{Hom}}

\DeclareMathOperator{\Ga}{\Gamma}
\DeclareMathOperator{\Rr}{\mathcal{R}}
\DeclareMathOperator{\tGa}{\tilde{\Gamma}}

\title[Riedtmann's functors]{On Riedtmann's well-behaved functors and applications to composites of irreducible morphisms}
\author[Chust]{Viktor Chust}
\address{(Viktor Chust) Institute of Mathematics and Statistics - University of São Paulo, São Paulo, Brazil}
\email{viktorchust.math@gmail.com }

\author[Coelho]{Fl\'avio U. Coelho}
\address{(Flávio U. Coelho) Institute of Mathematics and Statistics - University of São Paulo, São Paulo, Brazil}
\email{fucoelho@ime.usp.br}
\thanks{corresponding author: fucoelho@ime.usp.br}
\date{}

\begin{document}


\maketitle

\begin{abstract}
    In this survey, we summarize some results in the literature involving the mesh category, which is a combinatorial representation of the category of modules over a finite-dimensional associative algebra. We discuss Riedtmann's well-behaved functors, which compare the mesh category with the module category, and discuss how the properties of these functors can be applied to study the problem of composing irreducible morphisms, which is the problem of deciding when the composition of $n$ irreducible morphisms is non-zero and lies on the ($n+1$)-th power of the radical. 
\end{abstract}

\section*{Introduction}

The Auslander-Reiten theory is one of the main tools to study representations over Artinian algebras. By this theory, we define \textit{radical} and \textit{irreducible morphisms}. A morphism $f:X \rightarrow Y$ between two (finitely generated, right) modules over an Artinian algebra $A$
is called \textbf{radical} if for every section (split monomorphism) $j: X' \rightarrow X$ and retraction (split epimorphism) $p: Y \rightarrow Y'$, it holds that $pfj: X' \rightarrow Y'$ is not an isomorphism. We denote by $\rad_A(X,Y)$ the subspace of $\Hom_A(X,Y)$ formed by the radical morphisms. A morphism $f:X \rightarrow Y$ between two modules is called \textbf{irreducible} if:
(i) $f$ is neither a section nor a retraction, and (ii) if $f = hg$, where $g: X \rightarrow Z$ and $h: Z \rightarrow Y$ are morphisms, then either $g$ is a section or $h$ is a retraction.

In case $X$ and $Y$ are both indecomposable, the definition is simpler: a morphism $f: X \rightarrow Y$ is radical if and only if it is not an isomorphism, and it is irreducible if and only if it is radical and does not factor as the composite of two radical morphisms.

Since $\rad_A(X,Y)$ actually forms what is called an \textit{ideal} of the module category, we can consider the \textit{powers} of it, defining recursively: $\rad^1 = \rad$, $\rad^n = \rad^{n-1} \cdot \rad$ for $n>1$. We also denote by $\rad^{\infty}$ the intersection of all the ideals $\rad^n$ for $n \geq 1$. With this notation, if $X$ and $Y$ are indecomposable, a morphism $f: X \rightarrow Y$ is irreducible if and only if it belongs to $\rad_A(X,Y) \setminus \rad^2_A(X,Y)$.

That way, for $n \geq 1$, the composite of $n$ irreducible morphisms between indecomposable modules belongs to the $n$-th power of the radical. However, if $n>1$, it is not true in general that it does not belong to $\rad^{n+1}$, even if we assume it is non-zero. (By the way, Example 1.4 (a) in \cite{CCT1} brings two irreducible morphisms whose composite is non-zero and belongs to $\rad^{\infty}$). So a general problem (which we usually call here as the \textit{problem of composites of irreducible morphisms}) is to decide when the composition of $n$ irreducible morphisms is non-zero and belongs to $\rad^{n+1}$.

A classical result in that sense is the Igusa-Todorov Theorem (\cite{IT1}, also stated as Theorem~\ref{th:igusa todorov} below), which asserts that the composite of $n$ irreducible morphisms along a sectional path can never be in $\rad^{n+1}$. The problem of composites of irreducible morphisms probably begins to be addressed more explicitly in \cite{CCT1}, where the case $n=2$ is solved. We will recall such a result, along with other results about composites of irreducible morphisms obtained in recent years, in Section~\ref{subsec:background}.

While the results mentioned above solve particular cases of the problem of composites of irreducible morphisms, a general criterion that works for arbitrary lengths of compositions along arbitrary paths was introduced by the authors C. Chaio, P. Le Meur and S. Trepode in a series of articles (\cite{CT,CMT1,CMT2,CMT3}). This criterion is included as Theorem~\ref{th:inf ou atalhos} below. One of our main objectives in this survey is to explain the techniques they used, and then unify their results, summarizing some of the proofs involved.

The techniques that Chaio et al. used to prove this general criterion are non-elementary, and involve \textit{coverings of quivers, mesh categories} and \textit{well-behaved functors}, which were introduced in \cite{Rie,BG} over algebraically closed fields. Later, \cite{IT1} defined mesh categories over arbitrary fields. Coverings of quivers are defined in analogy to Algebraic Topology, while mesh categories are similar to the module category, but are defined in a more combinatorial manner and have better properties than the module category. Finally, the well-behaved functors relate the mesh category and the module category, allowing one to obtain results about the latter category using the former.

It is important to remark that we have made the decision here to rename `well-behaved functors', as usually found in literature, as `Riedtmann functors', which we hope will give these functors a more formal, precise name, also giving more credit for their creators.

While \cite{Rie,BG} introduce the concepts above, \cite{IT1,CMT1,CMT2} define them slightly different in nature, and we believe we have a couple of original results (Theorems~\ref{th:universal x generico} and~\ref{th:universal x generico 2}) towards proving that these multiple definitions are convergent in a certain manner. Showing these connections was in fact a key motivation for this survey, and these our main results.

This work is organized as follows: in Section~\ref{sec:prelim} we remind some preliminary concepts needed for the rest of the paper. In particular, \ref{subsec:ar} provides connections between irreducible morphisms and the so-called Auslander-Reiten theory, while in~\ref{subsec:background} we collect some of the background on composites of irreducible morphisms, especially regarding results in particular cases. We will explain the necessary techniques of coverings, mesh categories and Riedtmann functors in Sections~\ref{sec:coverings}, \ref{sec:mesh categories} and~\ref{sec:Riedtmann}, respectively. After that, we proceed for the applications in Section~\ref{sec:appl to composites}: in~\ref{subsec:inf ou atalhos} we use the concepts above to give the general criterion over composites (Theorem~\ref{th:inf ou atalhos}) due to Chaio, Le Meur and Trepode. And finally, in~\ref{subsec:new proof IT}, we also show how to derive a new proof of the Igusa-Todorov Theorem using these techniques.

We will always be using the letter $k$ to denote a field. By an algebra $A$, we mean, unless otherwise stated, a finite dimensional associative and unitary $k$-algebra. By modules we will usually mean finitely generated right $A$-modules. We refer to the books \cite{AC,AC2} for unexplained notions on modules and representation theory.

\section{Preliminary concepts}
\label{sec:prelim}

\subsection{Quivers}
\label{subsec:quiv}

A \textbf{quiver} $Q$ is a 4-uple $Q=(Q_0,Q_1,s,e)$, where $Q_0$ is a set of \textbf{vertices}, $Q_1$ is a set of \textbf{arrows}, and $s,e:Q_1 \rightarrow Q_0$ are two functions which give, respectively, the \textbf{start} and the \textbf{end} of an arrow. To each arrow $\alpha \in Q_1$ we can consider its formal inverse $\alpha^{-1}$, and we establish by convention that $s(\alpha^{-1}) = e(\alpha)$ and $e(\alpha^{-1}) = s(\alpha)$.

A \textbf{walk} over $Q$ is a sequence $\beta_n \cdots \beta_1$, where, for every $1 \leq i \leq n$, $\beta_i$ is either an arrow or the inverse of an arrow, and where $e(\beta_i) = s(\beta_{i+1})$ for every $1 \leq i < n$. If $w= \beta_n \cdots \beta_1$ is a walk, we extend the notations used for start and end vertices: $s(w) \doteq s(\beta_1)$ and $e(w) \doteq e(\beta_n)$. If $w$ and $w'$ are two walks and $e(w) = s(w')$, we can define the composition $w'w$ naturally by juxtaposition.

If $\beta_n \cdots \beta_1$ is a walk and all $\beta_i's$ are arrows (rather than possibly inverses of arrows), then we say that this walk is a \textbf{path of length $n$}. Additionally, one associates to each vertex $x$ of $Q$ a trivial \textbf{path of length 0}, denoted by $\epsilon_x$. Of course, $s(\epsilon_x)=e(\epsilon_x) =x$.

\subsection{Auslander-Reiten theory}
\label{subsec:ar}

We have written above about radical and irreducible morphisms, and now we give further reminders on Auslander-Reiten theory. This theory was introduced in the 1970's by M. Auslander and I. Reiten as a new technique to study the representations of Artin algebras. The key point is the existence of almost split sequences, a concept which we now recall. 

Let $A$ be an Artin algebra. A short exact sequence $(*): 0 \rightarrow L \xrightarrow{f} M \xrightarrow{g} N \rightarrow 0 $ is called an \textbf{almost split sequence} or an \textbf{Auslander-Reiten sequence} provided both $f$ and $g$ are irreducible morphisms. Since irreducible morphisms do not split, such sequences also do not split. Moreover, if $(*)$ is an almost split sequence, then it satisfies: (i) $L$ and $N$ are indecomposable; (ii) $g$ is \textbf{right minimal almost split}, that is, every $h \in \operatorname{End} M$ such that $gh = g$ is an automorphism (this is the \textit{right minimality} of $g$) and for every radical morphism $v: V \rightarrow N$ there exists $v': V \rightarrow M$ such that $v = gv'$; (iii) $f$ is \textbf{left minimal almost split}, that is, it satisfies the dual condition stated for $g$. 

Auslander and Reiten have shown that if $X$ is an indecomposable non-projective $A$-module, then there exists a unique almost split sequence $ 0 \rightarrow \tau X \xrightarrow{f} M \xrightarrow{g} X \rightarrow 0 $ ending at $X$. The indecomposable $A$-module $\tau X$, uniquely determined from $X$ up to isomorphism, is called the \textbf{Auslander-Reiten translation} of $X$. Dually, if $Y$ is an indecomposable non-injective $A$-module, then there exists a unique almost split sequence $ 0 \rightarrow Y \xrightarrow{f} M \xrightarrow{g} \tau^{-1} Y \rightarrow 0$ starting at $Y$. Moreover, any right (or left) minimal almost split morphism ending (starting, respectively) at an indecomposable non-projective (non-injective) module is the end (start) of an almost split sequence. It is worthwhile mentioning that for a given indecomposable projective module $P$, then the canonical inclusion $\rad P \rightarrow P$ is a right minimal almost split morphism, while for an indecomposable injective module $I$, the canonical projection $I \rightarrow I/\operatorname{soc} I$ is a left minimal almost split morphism.

Another connection between irreducible morphisms and almost split morphisms goes as follows. Let $X$ be an indecomposable module. Then a morphism $f: Y \rightarrow X$  (or $g : X \rightarrow Z $) is irreducible if and only if there exists a right (left, respectively) minimal almost split morphism $(f, f'): Y \oplus Y' \rightarrow X$ ($(g, g')^t : X \rightarrow Z \oplus Z' $, respectively). So, irreducible morphisms are described in terms of right and left minimal almost split morphisms. 

An important feature of this theory is the so-called Auslander-Reiten quiver. We shall recall it in the case where $A$ is a finite dimensional algebra over an algebraically closed field $k$. The \textbf{Auslander-Reiten quiver} (or \textbf{AR-quiver} for short) $\Gamma(\md A)$ of $A$ is a quiver defined as follows: (i) the vertices of  $\Gamma(\md A)$ are in bijection with the isomorphism classes of indecomposable $A$-modules; (ii) for two vertices $M, N$, the number of arrows $M\rightarrow N$ equals the dimension of the $k$-vector space $\operatorname{irr}_k(M,N) \doteq \rad_A(M,N)/\rad^2_A(M,N)$ of the irreducible morphisms from $M$ to $N$. The Auslander-Reiten quiver is interesting because it allows us to visualize the module category in the diagrammatic form of a quiver, while also summarizing information about the module category.

If $A$ is an indecomposable representation-finite algebra (i.e., an algebra having only finitely many indecomposable modules), then $\Gamma(\md A)$  is connected, but, in general, this is not true. Also, when equipped with the Auslander-Reiten translation $\tau$, $\Gamma(\md A)$ is a translation quiver, a concept defined below. For further details on Auslander-Reiten theory, we indicate \cite{AC}.

\subsection{Some background on the problem of compositions of irreducible morphisms}
\label{subsec:background}

As we have suggested in the Introduction, probably the first result that deals with composites of irreducible morphisms is due to K. Igusa and G. Todorov. If $X_0 \xrightarrow{h_1} X_1 \xrightarrow{h_2} \cdots \xrightarrow{h_n} X_n$ is a path of irreducible morphisms, where $X_0,\dots,X_n$ are indecomposable modules, then we say that this path is \textbf{sectional} if $\tau X_i \neq X_{i-2}$ for $2 \leq i \leq n$. (Where $\tau$ denotes the Auslander-Reiten translation). The result is the following:

\begin{teo}[Igusa, Todorov, \cite{IT1}]
\label{th:igusa todorov}
    If $X_0 \xrightarrow{h_1} X_1 \xrightarrow{h_2} \cdots \xrightarrow{h_n} X_n$ is a sectional path of irreducible morphisms between indecomposable modules, then $h_n \cdots h_1 \in \rad^n(X_0,X_n) \setminus \rad^{n+1}(X_0,X_n)$.
\end{teo}

Some years later, \cite{CCT1} gave a complete solution for the problem of composites of irreducible morphisms in the particular case where $n=2$, by establishing the following result:

\begin{teo}[\cite{CCT1},2.2]
\label{teo:cct1}

Given $X,Y,Z$ indecomposable modules over an Artin algebra $A$, the following are equivalent:

\begin{enumerate}
    \item There are irreducible morphisms $h:X \rightarrow Y$ e $h':Y \rightarrow Z$ such that $h'h \neq 0$ and $h'h \in \rad^3(X,Z)$.
    
    \item There are an almost split sequence $0 \rightarrow X \xrightarrow{f} Y \xrightarrow{g} Z \rightarrow 0$ and non-isomorphisms $\phi_1: X \rightarrow N$ and $\phi_2: N \rightarrow Z$, with $N$ indecomposable and not isomorphic to $Y$, such that $\phi_2 \phi_1 \neq 0$.
    
    \item There are an almost split sequence $0 \rightarrow X \xrightarrow{f} Y \xrightarrow{g} Z \rightarrow 0$ and a morphism $\phi \in \rad^2(Y,Y)$ such that $g\phi f \neq 0$.
    
    \item There is an almost split sequence $0 \rightarrow X \xrightarrow{f} Y \xrightarrow{g} Z \rightarrow 0$ and $\rad_A^4(X,Z) \neq 0$.
    
\end{enumerate}
\end{teo}

After this, many results about compositions of irreducible morphisms in particular cases have been obtained, and we highlight some of them:

\begin{itemize}
    \item In \cite{CCT3}, the authors solve the case $n=3$: similar to above, they give necessary and sufficient conditions for the existence of a non-zero composite of 3 irreducible morphisms between indecomposable modules belonging to $\rad^4$.

    \item In \cite{ChS}, results about the composition of 4 irreducible morphisms in $\rad^5$ are given.

    \item In \cite{CCT2}, the authors study compositions along what they define as \textit{almost sectional paths}, which are paths which may only fail to be sectional on the first and the last module, but not both.

    \item In \cite{AlC}, it is proved that if the composition of two irreducible morphisms between indecomposable modules belongs to $\rad^3$, then it must also belong to $\rad^5$.

    \item In \cite{C}, Corollary 2.19, it is proved that if the composite of 3 irreducible morphisms belongs to $\rad^4$, then it must belong to $\rad^6$.

    \item \cite{CPT} studies the problem of composites of irreducible morphisms over regular components, that is, componentes of the AR-quiver not containing projective nor injective modules. They prove, for example, that if the composite of $n$ irreducible morphisms between modules in a stable tube or in a component of type $\mathbb{ZA}_{\infty}$ lies in $\rad^{n+1}$, then actually it lies in $\rad^{\infty}$  (Corollary 2.10).

    \item Some articles (\cite{CCT1,CCT4,CGS}) also showed that this problem is particularly interesting to study over string algebras (see \cite{BR} for a definition): for example, they gave criteria for the existence of certain compositions of irreducible morphisms in terms of the possible configurations of the Gabriel quiver and the relations over it.

    \item Specifically, \cite{CGS} shows that for every $n \geq 3$ and every $m \geq 4$, there is a string algebra for which there is a path of $n$ irreducible morphisms $f_1, \dots, f_n$ between indecomposable modules whose composition $f_n \cdots f_1$ belongs to $\rad^{n+m} \setminus \rad^{n+m+1}$, and such that the compositions $f_n \cdots f_2$ and $f_{n-1} \cdots f_1$ do not belong to $\rad^n$.

    \item In a paper to appear (\cite{CC5}), we expand the result from last item, by obtaining that for every $n \geq 2$ and every $m \geq 3$, there is a string algebra for which there is a path of $n$ irreducible morphisms $f_1, \dots, f_n$ between indecomposable modules whose composition $f_n \cdots f_1$ belongs to $\rad^{n+m} \setminus \rad^{n+m+1}$.
\end{itemize}

Parallel to the results above, which deal with the category of modules, this problem has been extended for categories of complexes. For example, \cite{CST} gives results over the bounded derived category.

\subsection{Translation quivers}
\label{subsec:tr quivs}

The Auslander-Reiten quivers, as well as their connected components, are a particular case of what is called \textbf{translation quivers}. This concept has been introduced as a generalization, since many results about Auslander-Reiten quivers rely only on combinatorial aspects of these quivers and on the Auslander-Reiten translation $\tau$ viewed as simply a function. We proceed for the definition:

\begin{defi}
A \textbf{translation quiver} is a quiver $\Gamma$ such that:

\begin{enumerate}
    \item $\Gamma$ has no loops, that is, no arrows starting and ending at the same vertex;
    
    \item $\Gamma$ is equipped with a set of vertices called {\it projective vertices} and another set of vertices called {\it injective vertices};
    
    \item There is a bijection (called \textit{translation}) $\tau: x \mapsto \tau(x)$ between non-projective and non-injective vertices; and 
    
    \item  For each pair of vertices $x$ and $y$ of $\Gamma$, with $x$ non-projective, there is a bijection $\sigma: \alpha \mapsto \sigma(\alpha)$ between arrows of the form $y \rightarrow x$ and arrows of the form $\tau(x) \rightarrow y$.
\end{enumerate}
\end{defi}

Every component of an AR-quiver is an example of a translation quiver, with the Auslander-Reiten translation being the function $\tau$ above.

In general, we will always use the letters $\tau$ e $\sigma$ to denote the functions associated to a given translation quiver $\Gamma$. In addition to the conditions above, we shall be always assuming that every translation quiver is locally finite. That means that for every vertex $x$ of $\Gamma$, there is only a finite number of arrows starting at $x$ and a finite number ending at it. Note that the components of an AR-quiver are always locally finite as translation quivers.

We still have some additional definitions regarding translation quivers. Let $\Gamma$ be a translation quiver. If $x_0 \xrightarrow{\alpha_1} x_1 \xrightarrow{\alpha_2} \cdots \xrightarrow{\alpha_n} x_n$ is a path over $\Gamma$, we call it a \textbf{sectional path} if $x_i \neq \tau x_{i+2}$ for every $i$ between $0$ and $n-2$. We say that the translation quiver $\Gamma$ has \textbf{trivial valuation} if for every pair of vertices $x, y$ of $\Gamma$ there is at most one arrow $x \rightarrow y$ in $\Gamma$.

Also, given a non-projective vertex $x$ of $\Gamma$, the \textbf{mesh} ending at $x$ is the full subquiver of $\Gamma$ determined by all the arrows that end at $x$ and the arrows that start at $\tau x$.

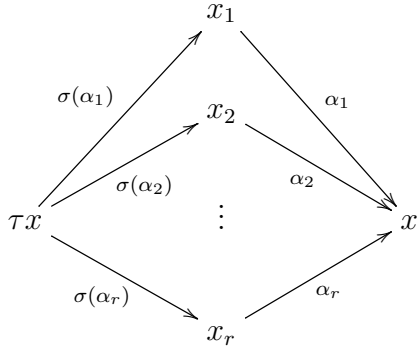
    \begin{displaymath}
        \xymatrix{ && x_1 \ar[ddrr]^{\alpha_1}&& \\
        && x_2 \ar[drr]_{\alpha_2}&&\\
        \tau x \ar[uurr]^{\sigma(\alpha_1)} \ar[urr]_{\sigma(\alpha_2)} \ar[drr]_{\sigma(\alpha_r)}&& \vdots && x \\
        &&x_r \ar[urr]_{\alpha_r} &&}
    \end{displaymath}

\section{Coverings, homotopy and the universal covering}
\label{sec:coverings}

We now recall the concept of coverings, introduced in the Representation Theory of Algebras in the 1970's, initially for the study of representation-finite algebras. 

\begin{defi}[\cite{Rie,BG}]
Let $\Delta$ and $\Gamma$ be two translation quivers. We say that a quiver morphism $\pi: \Delta \rightarrow \Gamma$ is a \textbf{covering} if:

\begin{enumerate}
    \item A vertex $x$ of $\Delta$ is projective (or injective, respectively) if and only if $\pi(x)$ is also projective (or injective, respectively).
    
    \item $\pi$ commutes with the translation functions in $\Delta$ and $\Gamma$ (wherever these are defined). 
    
    \item For every vertex $x$ of $\Delta$, the map $\alpha \mapsto \pi(\alpha)$ induces a bijection from the set of arrows of $\Delta$ starting at $x$ (or ending at $x$, respectively) and the set of arrows of $\Gamma$ starting at $\pi(x)$ (or ending at $\pi(x)$, respectively). 
\end{enumerate}
\end{defi}

As one can see from the definition above, a covering of a translation quiver $\Gamma$ is another translation quiver $\Delta$ which locally resembles $\Gamma$, much like the definition of covering spaces in Algebraic Topology.

Let $\Gamma$ be a translation quiver. Also in analogy to what is done in Algebraic Topology, in \cite{BG} the authors Bongartz and Gabriel introduced the concept of \textbf{homotopy} between walks over $\Gamma$. Next we explain this idea:

Consider the smallest equivalence relation $\sim$ between walks over $\Gamma$ which satisfies conditions a),b) and c) below:

\begin{enumerate}
    \item [a)] If $\alpha: x \rightarrow y$ is an arrow of $\Gamma$, then $\alpha \alpha^{-1} \sim \epsilon_y$ and $\alpha^{-1} \alpha \sim \epsilon_x$. (Recall that $\epsilon_x$ and $\epsilon_y$ denote the paths of length zero over the vertices $x$ and $y$ respectively).
    
    \item [b)] If $x$ is a non-projective vertex and $\alpha: z \rightarrow x$, $\alpha': z' \rightarrow x$ are two arrows ending at $x$, then $\alpha (\sigma \alpha) \sim \alpha' (\sigma \alpha')$.
    
    \item [c)] If $\gamma_1,\gamma,\gamma',\gamma_2$ are walks over $\Gamma$ such that $\gamma \sim \gamma'$ and the compositions $\gamma_1 \gamma \gamma_2$ e $\gamma_1 \gamma' \gamma_2$ are defined, then $\gamma_1 \gamma \gamma_2 \sim \gamma_1 \gamma' \gamma_2$.
\end{enumerate}

The relation $\sim$ is then called a \textbf{homotopy} between walks.

Also as in Algebraic Topology, in \cite{BG} the \textbf{universal covering} $\tilde{\Gamma}$ of a translation quiver $\Gamma$ is defined. This covering is defined as follows:

Fix a vertex $x_0$ of $\Gamma$, which we suppose to be connected. (The quiver $\tilde{\Gamma}$ will not depend on the choice of $x_0$, up to isomorphism).

\begin{itemize}
    
    \item The vertices of $\tilde{\Gamma}$ are the classes of homotopy $\overline{w}$ of walks $w$ over $\Gamma$ which start at the given vertex $x_0$ and end in some other (non-fixed) vertex, which will be $e(w)$. (Note that $e(w)$ does not depend on the choice of the representative $w$ of the class of homotopy $\overline{w}$).

    \item The arrows of $\tilde{\Gamma}$ are pairs $(\overline{w},\alpha)$, where $\overline{w}$ is a class of homotopy and $\alpha:e(w) \rightarrow z$ is an arrow in $\Gamma$. The arrow $(\overline{w},\alpha)$ starts at $\overline{w}$ and ends at $\overline{\alpha w}$.

    \item If $w$ is a walk and $e(w)$ is projective, then $\overline{w}$ is projective. If $e(w)$ is not projective, there are arrows $\tau(e(w)) \xrightarrow{\sigma \alpha} z \xrightarrow{\alpha} e(w)$ and so we define $\tau(\overline{w}) = \overline{(\sigma \alpha)^{-1}\alpha^{-1} w}$. (This also does not depend on the choice of $\alpha$).

    \item We define the covering of quivers $\pi: \tilde{\Gamma} \rightarrow \Gamma$, given by $\pi(\overline{w}) = e(w)$.
\end{itemize}

With that, the pair $(\tilde{\Gamma},\pi)$ (or simply $\tilde{\Gamma}$) is the \textbf{universal covering} of $\Gamma$.

In \cite{CMT1} another kind of covering is introduced, the so-called \textbf{generic covering}, which can be more interesting to study than the universal one in the case where $\Gamma$ has non-trivial valuation. In \cite{CMT1} the generic covering is also denoted by $\tilde{\Gamma}$ but here we shall denote it by $\hat{\Gamma}$ to differentiate from the universal covering.

The generic covering $\hat{\Gamma}$ is defined as follows: consider the smallest equivalence relation $\sim'$ between the walks over $\Gamma$ the satisfies conditions a), b) and c) above (with $\sim'$ instead of $\sim$) and also condition d) below:

\begin{enumerate}
    \item [d)] If $\alpha,\beta: x \rightarrow y$ are two arrows with the same start and end vertex, then $\alpha \sim' \beta$.
\end{enumerate}

The relation $\sim'$ has also been called a \textbf{homotopy} between walks. Then the quiver $\hat{\Gamma}$ is defined in the same way as $\tilde{\Gamma}$, but switching $\sim$ by $\sim'$.

\section{Mesh categories}
\label{sec:mesh categories}

\subsection{The general case}
\label{subsec:mesh general}

Let $k$ be field (which initially could be any field, as in \cite{IT1}) and let $\Gamma$ be a translation quiver. In order to define the mesh category over $\Gamma$, first we need an auxiliary concept, that of \textbf{$k$-modulation}:

\begin{defi}[\cite{IT1}]
    A \textbf{$k$-modulation} over a translation quiver $\Gamma$ consists of the following structure:

    \begin{itemize}
        \item  For every vertex $x$ of $\Gamma$, a finite-dimensional division $k$-algebra $k_x$.

        \item For every arrow $x \rightarrow y$ of $\Gamma$, a $(k_x - k_y)$-bimodule of finite $k$-dimension $M(x,y)$. 

        \item For every non-projective vertex $x$ of $\Gamma$, an isomorphism of $k$-algebras $\tau_*: k_x \xrightarrow{\sim} k_{\tau x}$.

        \item For every non-projective vertex $x$ of $\Gamma$ and every vertex $y$ such that there is an arrow $y \rightarrow x$, a non-degenerate $k_y$-bilinear form $\sigma_*: M(y,x) \otimes_{k_x} M(\tau x,y) \rightarrow k_y$ (where $k_x$ acts on the left of $M(\tau x,y)$ using the isomorphism $\tau_*: k_x \xrightarrow{\sim} k_{\tau x}$).
    \end{itemize}

With that structure, $\Gamma$ is called a \textbf{modulated translation quiver}.
\end{defi}

\begin{obs}
    If $\pi: \Delta \rightarrow \Gamma$ is a covering of translation quivers and $\Gamma$ is equipped with a $k$-modulation, is easy to see how to define a $k$-modulation in $\Delta$ using $\pi$: for instance, to every vertex $x \in \Delta_0$ we associate $k_{\pi x}$ and for every arrow $x \rightarrow y$ we associate $M(\pi x, \pi y)$. With this we say that $\Delta$ has an \textbf{$k$-modulation induced} by $\Gamma$ (and $\pi$).
\end{obs}

Consider now the case where $\Gamma$ is a component of the AR-quiver of an algebra $A$. In this case, it admits a $k$-modulation as described below, which is called \textbf{standard $k$-modulation} over $\Gamma$:

\begin{itemize}
    \item If $X$ is an indecomposable module from $\Gamma$, we make $k_X = \operatorname{End_A} X/\rad_A(X,X)$.
    
    \item If $X \rightarrow Y$ is an arrow from $\Gamma$, we make $M(X,Y) = \rad_A(X,Y)/\rad^2_A(X,Y) \doteq \operatorname{irr}(X,Y)$.

    \item If $X$ is indecomposable non-projective in $\Gamma$, we define the function $\tau_*: \operatorname{End_A} X/\rad_A(X,X) \rightarrow \operatorname{End_A} \tau X/\rad_A(\tau X,\tau X)$ the following way: fix $0 \rightarrow \tau X \rightarrow E \rightarrow X \rightarrow 0$ an almost split sequence ending at $X$. If $u: X \rightarrow X$ is an endomorphism of $X$, we define $\tau_*(\overline{u})$ as being the class $\overline{v}$, where $v: \tau X \rightarrow \tau X$ is a morphism that makes the following diagram commute:

    \begin{displaymath}
        \xymatrix{0 \ar[r] & \tau X \ar[d]_v \ar[r]& E \ar@{=}[d] \ar[r]& X \ar[d]^u \ar[r]& 0 \\ 
        0 \ar[r]& \tau X \ar[r]& E \ar[r]& X \ar[r]& 0}
    \end{displaymath}

    \item If $Y \rightarrow X$ is an arrow of $\Gamma$ and $X$ is not projective, fix $0 \rightarrow \tau X \rightarrow E \rightarrow X \rightarrow 0$ an almost split sequence ending at $X$. If $\overline{u} \in \rad(Y,X) / \rad^2(Y,X)$ e $\overline{v} \in \rad(\tau X,Y) / \rad^2(\tau X,Y)$ are classes of morphisms, we want to define $\sigma_*(\overline{u} \otimes \overline{v}) \in k_Y$. We take $\sigma_*(\overline{u} \otimes \overline{v}) = \overline{v'u'}$, where $u'$ and $v'$ are morphisms that make the following diagram commute:

    \begin{displaymath}
        \xymatrix{&&& Y \ar[dl]_{u'} \ar[d]^u&\\
        0 \ar[r]& \tau X \ar[d]_v \ar[r]& E \ar[dl]^{v'} \ar[r]& X \ar[r]& 0\\
        & Y &&&}
    \end{displaymath}
\end{itemize}

Of course one needs to prove that these are well-defined (for example, that it does not depend on the choice of almost split sequences). We recommend \cite{IT1} for more details. 

Now fix $\Gamma$ a modulated translation quiver. Let $k\Gamma$ be the category whose objects are the vertices of $\Gamma$, and for $x,y \in \Gamma$, the morphisms $x \rightarrow y$ are given by $k\Ga(x,y) = \bigoplus_{i \in \mathbb{N}} (k\Ga)_i(x,y)$, where $(k\Ga)_i(x,y)$ can be defined recursively:

\begin{itemize}
    \item $(k\Ga)_0(x,y) = \begin{cases}
        k_x \text{, if } x=y \\
        0 \text{, if } x \neq y
    \end{cases}$

    \item for $i > 0$, $(k\Ga)_i(x,y) = \bigoplus_{z \in y^-} (k\Ga)_{i-1}(x,z) \otimes M(z,y)$ (where $y^-$ denotes the set of immediate predecessors of $y$).
\end{itemize}

The category $k\Gamma$ is then called the \textbf{path category} over $\Gamma$. As observed in \cite{CMT2}, the path category is a tensor category, as follows: let $S$ be the semisimple category whose objects are the vertices of $\Gamma$, and for every $x,y \in \Gamma_0$, the morphisms $S(x,y)$ are defined by: $S(x,y) = k_x$ if $x = y$ and $S(x,y) = 0$ if $x \neq y$. Then the map $(x,y) \mapsto M(x,y)$ defines a $(S-S)$-bimodule $M$ over the category $S$. The tensor category $T_S(M)$ then coincides with the path category $k\Gamma$ over $\Gamma$. 

We now proceed to the definition of mesh category, which will be a quotient of the path category. Let us tell what are the relations over which we wish to take the quotient. Let $x$ be a non-projective vertex of $\Gamma$. If $y$ is a vertex such that there is an arrow $y \rightarrow x$ in $\Gamma$, fix $\{u_1,\cdots,u_d\}$ a $k_y$-basis of $M(y,x)$. Since the form $\sigma_*$ associated to $y$ and $x$ is non-degenerate, it induces a non-isomorphism $M(\tau x,y) \cong \Hom_{k_y}(M(y,x),k_y)$, and so we have a corresponding dual basis $\{u_1^*,\cdots,u_d^*\}$ of $M(\tau x,y)$ (i.e., $\sigma_*(u_i^* \otimes u_j) = \delta_{ij}$, where $\delta_{ij}$ is Kronecker's delta). Then we define $\gamma_x = \sum_{y \in x^-} \sum_{i=1}^d u_i^* \otimes u_i \in k\Gamma(\tau x, x)$, which is the \textbf{mesh relation} ending at $x$. (Again see \cite{IT1} for the proof that $\gamma_x$ is independent of basis choice).

Let then $I$ be the ideal of the category $k\Gamma$ generated by the morphisms $\gamma_x: \tau x \rightarrow x$, where $x$ runs through the non-projective vertices of $\Gamma$. The \textbf{mesh category} over $\Gamma$ is then defined (\cite{IT1,CMT2}) as being the quotient category $k(\Gamma) \doteq k\Gamma/I$.

\begin{obs}
\label{obs:categoria graduada}
    The path category $k\Ga$ is \textbf{$\mathbb{N}$-graded}: that is, for every pair of vertices $x,y \in \Gamma_0$, we have $k\Ga(x,y) = \bigoplus_{i \in \mathbb{N}} (k\Ga)_i(x,y)$, and if $x,y,z \in \Gamma_0$, $(k\Ga)_i(y,z) \cdot (k\Ga)_j(x,y) \subseteq (k\Ga)_{i+j}(x,z)$ for every $i,j \in \mathbb{N}$. We say that the elements from a space $(k\Ga)_i(x,y)$ are \textbf{homogeneous elements (or morphisms) of degree} $i$.

    In that sense, mesh relations are homogeneous elements of degree 2, and so the ideal $I$ generated by them is a \textbf{homogeneous ideal}. Thus the mesh category $k(\Ga)$ is also $\mathbb{N}$-graded, for being the quotient of the $\mathbb{N}$-graded $k\Ga$ by a homogeneous ideal: for every pair of objects $x,y \in \Gamma_0$, we have a decomposition $k(\Ga)(x,y) = \bigoplus_{i \in \mathbb{N}} k(\Ga)_i (x,y)$, where each $k(\Ga)_i (x,y)$ is the space of equivalence classes of $(k\Ga)_i (x,y)$.
\end{obs}

Keeping the notation from Remark~\ref{obs:categoria graduada}, we define in the mesh category $k(\Ga)$ the ideal
$\Rr k(\Ga)$ given by $\Rr k(\Ga)(x,y) \doteq \bigoplus_{i \geq 1} k(\Ga)_i (x,y)$, which coincide with the ideal generated by the spaces $M(x,y)$ of morphisms of degree 1. The ideal $\Rr k(\Ga)$ is used to be called in literature as the \textbf{radical of the mesh category} by analogy with the Jacobson radical, even though that carries some abuse of language since it is not technically a radical in category-theoretic sense.

Since $\Rr k(\Ga)$ is an ideal, we can consider its powers: for every of vertices $x,y \in \Gamma_0$ and every $n \geq 1$, we define $\Rr^n k(\Ga) = \bigoplus_{i \geq n} k(\Ga)_i(x,y)$.

\subsection{The particular case where $k$ is algebraically closed}
\label{subsec:mesh ac}

Originally (in \cite{Rie,BG}) mesh categories used to be defined only in the case where the base field $k$ is algebraically closed, and the definition was different from the one we gave above. Let us give more details.

Let $k$ be an algebraically closed field and let $\Gamma$ be a translation quiver. Differently from the general case where $k$ may be any field, we do not need $k$-modulations to define the path category $k\Ga$ and the mesh category $k(\Ga)$, since $k$-modulations are trivialized when $k$ is algebraically closed. Namely, if $x$ is a vertex of $\Gamma$ to which we associate a finite-dimensional division $k$-algebra $k_x$ as in the definition of modulation, then $k_x$ coincides with $k$, for being an algebraic (thus finite) extension of $k$ algebraically closed. If $x \rightarrow y$ is an arrow of $\Gamma$, a $(k_x - k_y)$-bimodule $M(x,y)$ is the same as a $k$-vector space $M(x,y)$.

The definition of these categories amounts to something simpler in this case. Let us follow the approach in \cite{Rie,BG,CMT1}.

The \textbf{path category} of $\Gamma$ is defined as the category $k\Gamma$ whose objects are the vertices $\Gamma$ and the morphisms between two vertices $x$ and $y$ are the elements of the vector space given by formal linear combinations of paths over $\Gamma$ which go from $x$ to $y$. 

The \textbf{mesh category} over $\Gamma$, denoted by $k(\Gamma)$, is the quotient category of the path category $k\Gamma$ by the ideal generated by all meshes in $\Gamma$, that is, generated by all morphisms of the form $\sum \alpha (\sigma \alpha): \tau x \rightarrow x$, where $x$ is a non-projective vertex of $\Gamma$ and the summation runs through all arrows $\alpha$ that end in $x$.

Here some ideas we had in the general case still apply:

\begin{itemize}
    \item The path category $k\Ga$ is \textbf{$\mathbb{N}$-graded}, the mesh relations are homogeneous elements of degree 2, the ideal generated by the mesh relations in homogeneous and so the mesh category $k(\Ga)$ is $\mathbb{N}$-graded.

    \item We can define the \textbf{radical of the mesh category} $k(\Ga)$ as being the ideal $\Rr k(\Ga)$ generated by the morphisms of degree 1, which are the equivalence classes of arrows in $\Gamma$.

    \item We also consider the powers of the ideal $\Rr k(\Ga)$, which are defined recursively: $\Rr^1 k(\Ga) = \Rr k(\Ga)$, and for $n \geq 1$, $\Rr^n k(\Ga) = \Rr^{n-1} k(\Ga)\cdot \Rr k(\Ga)$. Also note that $\Rr^n k(\Ga)$ coincides with the ideal generated by the classes of the paths having length greater than or equal to $n$.
\end{itemize}

\subsection{Connecting the two cases}
\label{subsec:mesh connection}

So mesh categories can be defined in two different ways when $k$ is algebraically closed. To show that they are convergent is the subject of our first main result here, which will be Theorem~\ref{th:universal x generico} below.

In order to make the notations precise, let us agree with the following conventions: if $\Gamma$ is a translation quiver, $\tilde{\Gamma}$ will denote the universal covering of $\Gamma$, $\hat{\Gamma}$ will denote the generic covering, $k_g(\Gamma)$ will denote the mesh category over $\Gamma$ as defined in~\ref{subsec:mesh general} (where $k$ is any and we fix a $k$-modulation over $\Gamma$) and $k_a(\Gamma)$ will denote the mesh category over $\Gamma$ as defined above in~\ref{subsec:mesh ac} (where we supposed $k$ is algebraically closed and we did not need $k$-modulations).

Note that $\Gamma$ may have multiple arrows between two vertices. So denote by $\Gamma^v$ the translation quiver obtained by replacing parallel multiple arrows in $\Gamma$ by a single arrow. Namely, $\Gamma^v$ has the same vertices of $\Gamma$, and if $x,y \in (\Gamma^v)_0$, there will be a unique arrow $x \rightarrow y$ in $\Gamma^v$ if and only if there is (at least) an arrow $x \rightarrow y$ in $\Gamma$.

The following theorem now gives the aforementioned connection between the two definitions of mesh category:

\begin{teo}
\label{th:universal x generico}

    Assume $k$ is algebraically closed. Let $\Gamma$ be a component of the AR-quiver of an algebra $A$, and assume $\Gamma^v$ has the standard $k$-modulation. Let $\pi: \Delta \rightarrow \Gamma$ be a covering of $\Gamma$, with $\Delta^v$ having the modulation induced from the one of $\Gamma^v$. With these notations:

    \begin{enumerate}
        \item $\widetilde{\Gamma^v} = (\hat{\Gamma})^v$;

        \item There is an isomorphism of $k$-linear graded categories $k_a(\Delta) \cong k_g(\Delta^v)$. (In particular $k_a(\hat{\Gamma}) \cong k_g(\widetilde{\Gamma^v})$).

    \end{enumerate}
\end{teo}

\begin{proof}
    \begin{enumerate}
        \item The proof is direct and left to the reader.

        \item The idea for proving the existence of the isomorphism $\iota$ is to prove that the path categories $k_g \Delta^v$ and $k_a \Delta$ are isomorphic and then observe that this isomorphism preserves mesh relations. From this a new isomorphism $k_g(\Delta^v) \cong k_a(\Delta)$ will be induced.

    Observe that $k_g \Delta^v$ and $k_a \Delta$ are both tensor categories with the same class of objects, thus in order to prove that they are isomorphic it is sufficient to prove that there are bijections between the sets that generate the space of morphisms.

    Let $o(\Delta^v)$ be the orbit graph of $\Delta$: this is the graph whose vertices are the orbits of the action of the translation $\tau$ over $\Delta$ and such that there is only one edge $o(x) \text{---} o(y)$ between the orbits of two vertices $x,y\in \Delta_0$ if and only if there are $m,n \in \mathbb{Z}$ such that $\tau^m x = \tau^n y$. Let $\theta$ be a function defined the following way: if $\alpha: x \rightarrow y$ is an arrow of $\Delta^v$, then there is only one edge $l: o(x) \text{---} o(y)$ in the graph $o(\Delta^v)$ and we define $\theta(\alpha) = l$.

    For each edge $l: X \text{---} Y$, choose an arrow $\alpha_l: x \rightarrow y$ such that $\theta(\alpha_l) = l$. (Note that $X = o(x)$ and $Y = o(y)$). Now choose irreducible morphisms $u_1,\cdots,u_d:\pi x \rightarrow \pi y$ in such a way that $\{\overline{u_1},\cdots,\overline{u_d}\}$ forms a basis of $\operatorname{irr}(\pi x, \pi y) = M(x,y)$.

    Then either $\alpha_l$ is the only arrow whose image under $\theta$ is $l$ or then, choosing another $\alpha_l$ if necessary, we can suppose $\sigma \alpha_l$ is defined. In that case there will be irreducible morphisms $v_1, \cdots, v_d: \tau \pi y \rightarrow \pi x$ such that $\{\overline{v_1},\cdots,\overline{v_d}\}$ is a basis of $\operatorname{irr}(\tau \pi y, \pi x)$, and $(\overline{v_1},\cdots,\overline{v_d}) =(\overline{u_1}^*,\cdots,\overline{u_d}^*)$ is the dual basis of $(\overline{u_1},\cdots,\overline{u_d})$ (considering the isomorphism $\operatorname{irr}(\tau \pi y, \pi x) \cong \operatorname{irr}(\pi x, \pi y)^*$ given by the $k$-modulation).

    Let now $\alpha: x \rightarrow y$ be an arrow in $\Delta^v$. then there is $r \in \mathbb{Z}$ such that $\alpha = \sigma^r \alpha_{\theta(\alpha)}$. If $r$ is even, we can identify the set of arrows from $x$ to $y$ in $\Delta$ with the elements of the set $\{\overline{u_1},\cdots,\overline{u_d}\}$, and in that case we define $\iota:k_g \Delta^v_1 (x,y)\rightarrow k_a \Delta_1 (x,y)$ as being the only $k$-linear map that takes $\{\overline{u_1},\cdots,\overline{u_d}\}$ to $\{\overline{u_1},\cdots,\overline{u_d}\}$. In case $r$ is odd, we identify the arrows from $x$ to $y$ with the elements of the set $\{\overline{v_1},\cdots,\overline{v_d}\}$, and we define $\iota:k_g \Delta^v_1 (x,y)\rightarrow k_a \Delta_1 (x,y)$ as being the only $k$-linear map that takes $\{\overline{v_1},\cdots,\overline{v_d}\}$ to $\{\overline{v_1},\cdots,\overline{v_d}\}$.

    It is clear that every $\iota:k_g \Delta^v_1 (x,y)\rightarrow k_a \Delta_1 (x,y)$ thus defined is a $k$-linear isomorphism, and that concludes the proof that $k_g\Delta^v \cong k_a \Delta$. Let us see that relations are preserved. 
    
    Let $y \in \Delta_0$ be a non-projective vertex, and let $y_1,\cdots, y_r$ be the immediate predecessors of $y$. For $i$ between 1 and $r$, note that, from the construction above, we have already chosen the arrows $\alpha_{i1}, \cdots, \alpha_{id_i}: y_i \rightarrow y$: with the notation above, they have either the form $\overline{u_{i1}},\cdots,\overline{u_{id_i}}$ or the form $\overline{v_{i1}},\cdots,\overline{v_{id_i}}$. In both cases we have that the arrows $\tau y \rightarrow y_i$ form a set $\{\alpha_{i1}^* \cdots \alpha_{id_i}^*\}$ dual to $\{\alpha_{i1} \cdots \alpha_{id_i}\}$, in the first case because we have chosen the bases like that, and in the second case because we use that $V^{**} \cong V$ for every finite-dimensional vector space $V$.

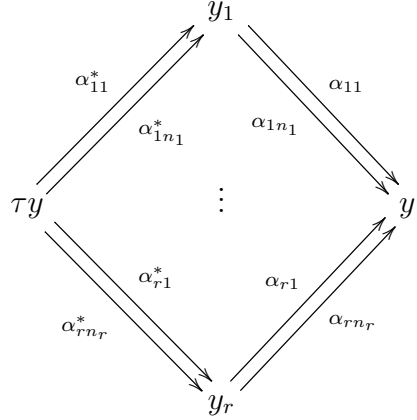
    \begin{displaymath}
        \xymatrix{ && y_1 \ar@<.5ex>[ddrr]^{\alpha_{11}} \ar@<-.5ex>[ddrr]_{\alpha_{1n_1}} && \\
        &&&& \\
        \tau y \ar@<.5ex>[uurr]^{\alpha_{11}^*} \ar@<-.5ex>[uurr]_{\alpha_{1n_1}^*} \ar@<.5ex>[ddrr]^{\alpha_{r1}^*} \ar@<-.5ex>[ddrr]_{\alpha_{rn_r}^*} && \vdots && y\\
        &&&& \\
        && y_r \ar@<.5ex>[uurr]^{\alpha_{r1}} \ar@<-.5ex>[uurr]_{\alpha_{rn_r}} && }
    \end{displaymath}

    Observe that, with these notations, the mesh relation that ends in $y$ is $\sum_{i=1}^r \sum_{j=1}^{n_i} \alpha_{ij} \alpha_{ij}^*$ either at the definition of $k_g(\Delta^v)$ or at the definition $k_a(\Delta)$, 
    proving that the isomorphism we have built between the path categories preserves mesh relations and thus concluding the proof of this item.
    \end{enumerate}
\end{proof}

\section{Riedtmann functors}
\label{sec:Riedtmann}

\subsection{Strongly Riedtmann functors}

Having defined the mesh category over a translation quiver $\Gamma$, we proceed to the definition of the comparisons between the mesh category and the module category, which have been called \textbf{well-behaved functors} in the literature, although we will be renaming them as \textbf{Riedtmann functors}, as explained in the introduction. 

From now on, we need to make additional assumptions on the base field $k$. We will assume that $k$ is \textbf{perfect}, that is, every field extension of $k$ is separable. The reader may notice that this condition on $k$ is not strictly necessary to state the definition of Riedtmann functors. However, without it the existence of these functors could hardly be expected. The hypothesis of $k$ being perfect plays its role in two ways: the first is through the following version of the Wedderburn-Malcev Theorem:

\begin{teo}[Wedderburn-Malcev]
    \label{th:wed-malcev}

    Let $k$ be a perfect field, and let $\Lambda$ be a finite dimensional $k$-algebra. Then the canonical projection $\pi: \Lambda \rightarrow \Lambda/\rad \Lambda$ admits a section $\iota: \Lambda/\rad \Lambda \rightarrow \Lambda$, i.e., $\pi \circ \iota = \operatorname{Id}$. 
\end{teo}

\begin{proof}
    See, e. g. , \cite{Pierce}, $\S$ 11.6.
\end{proof}

The second is in the following result:

\begin{teo}
\label{th:prod tensorial eh semissimples}
    Let $k$ be a perfect field. If $A$ and $B$ are two finite dimensional division $k$-algebras, then the tensor product $A \otimes_k B$ is a semisimple algebra.
\end{teo}

\begin{proof}
    See, e. g., \cite{Pierce}, Corollary on page 188 and Corollary b on page 192.
\end{proof}

The definition of Riedtmann's well-behaved functors will follow the approach from \cite{CMT2}, which is where this concept was defined for perfect fields. (Originally, in \cite{Rie} and \cite{BG}, the definition was restricted to algebraically closed fields).

Let $\Gamma$ be a component of the AR-quiver of an algebra $A$, and let $X$ be an indecomposable module of $\Gamma$. Then $k_X \doteq \operatorname{End}_A(X)/\rad_A(X,X)$ is a division algebra. 

By Theorem~\ref{th:wed-malcev}, there is a section $k_X = \operatorname{End}_A(X)/\rad_A(X,X) \hookrightarrow \operatorname{End}_A(X)$ of the canonical projection $\operatorname{End}_A(X) \twoheadrightarrow \operatorname{End}_A(X)/\rad_A(X,X)$. Let then $\Bbbk_X$ be the image of this section. Then $\Bbbk_X$ is a subalgebra of $\operatorname{End}_A(X)$ and there is an isomorphism $\Bbbk_X \simeq k_X$ via the restriction of the canonical projection $\operatorname{End}_A(X) \twoheadrightarrow k_X$. 

\begin{defi}[\cite{CMT2}]
    With the notations above, we say that $\Bbbk_X$ is a \textbf{section} of $k_X$.
\end{defi}

Let now $X \rightarrow Y$ be an arrow of $\Gamma$. Fix a section $\Bbbk_X$ of $k_X$ and a section $\Bbbk_Y$ of $k_Y$. Since $\Bbbk_X \simeq k_X$ and $\Bbbk_Y \simeq k_Y$, it is easy to transform $\operatorname{irr}(X,Y) \doteq \rad_A(X,Y)/\rad^2_A(X,Y)$, which is a $(k_X - k_Y)$-bimodule, into a $(\Bbbk_X - \Bbbk_Y)$-bimodule.

Since $k_X$ and $k_Y$ are division algebras, $\Bbbk_X \simeq k_X$ and $\Bbbk_Y \simeq k_Y$, by Theorem~\ref{th:prod tensorial eh semissimples}, we obtain that the algebra $\Bbbk_X^{op} \otimes_k \Bbbk_Y$ is semisimple. Thus the canonical projection $\rad_A(X,Y) \twoheadrightarrow \operatorname{irr}(X,Y)$ between $\Bbbk_X^{op} \otimes_k \Bbbk_Y$-modules admits a section $\iota: \operatorname{irr}(X,Y) \rightarrow \rad_A(X,Y)$.

\begin{defi}[\cite{CMT2}]
    With the notations above, $\iota: \operatorname{irr}(X,Y) \rightarrow \rad_A(X,Y)$ is a \textbf{$(\Bbbk_X - \Bbbk_Y)$-linear section} of $\operatorname{irr}(X,Y)$.
\end{defi}

With the terminology introduced above, we are ready for the definition of well-behaved functors:

\begin{defi}[\cite{CMT2}]
Let $k$ be a perfect field and let $\Gamma$ be a component of the AR-quiver of a $k$-algebra $A$. Let $\pi:\Delta \rightarrow \Gamma$ be a covering of $\Gamma$. Assume $\Gamma$ is equipped with the standard $k$-modulation and that $\Delta$ is equipped with the $k$-modulation induced from $\Gamma$. In these conditions, a $k$-linear functor $F:k(\Delta) \rightarrow \ind \Gamma$ is called \textbf{(strongly) well-behaved} or \textbf{(strongly) Riedtmann} if it satisfies the following conditions:

\begin{enumerate}
    \item For every vertex $x$ of $\Delta$, $Fx = \pi x$.

    \item For every vertex $x$ of $\Delta$, the morphism  of $k$-algebras $k_x \doteq k_{\pi x} \rightarrow \operatorname{End}_A(\pi x)$ given by $u \mapsto F(u)$ is a section of the canonical projection $\operatorname{End}_A(\pi x) \twoheadrightarrow k_{\pi x}$. We denote the image of that section by $\Bbbk_x$.

    \item For every arrow $x \rightarrow y$ of $\Delta$, the composite $k$-linear map
    
    $$\operatorname{irr}(\pi x, \pi y) = M(x,y) \hookrightarrow k(\Delta)(x,y) \xrightarrow{F} \rad_A(\pi x, \pi y)$$
    
    is a $(\Bbbk_x - \Bbbk_y)$-linear section of the canonical projection.
\end{enumerate}
\end{defi}

\subsection{Weakly Riedtmann functors}

In case $k$ is algebraically closed (which implies that $k$ is perfect), the definition of Riedtmann's well-behaved functors was made earlier than in \cite{CMT2}. Let us state the definition in this case.

\begin{defi}[\cite{CMT1}]
Suppose $k$ is algebraically closed. Let $\Gamma$ be a component of the AR-quiver of $A$ and let $\pi: \Delta \rightarrow \Gamma$ be a covering of quivers. A $k$-linear functor $F:k(\Delta) \rightarrow \ind \Gamma$ is called \textbf{(weakly) well-behaved} or \textbf{(weakly) Riedtmann} if the following conditions are verified for every vertex $x$ of $\Delta$:

\begin{enumerate}
    \item $Fx = \pi x$
    \item if $\alpha_1: x \rightarrow x_1, \cdots, \alpha_r: x \rightarrow x_r$ are all the arrows in $\Delta$ that start at $x$, then $[F(\overline{\alpha_1}) \cdots F(\overline{\alpha_r})]^t: Fx \rightarrow Fx_1 \oplus \cdots \oplus Fx_r$ is a source morphism (i.e., a left minimal almost split morphism).
    
    \item if $\alpha_1: x_1 \rightarrow x, \cdots, \alpha_r: x_r \rightarrow x$ are all the arrows in $\Delta$ that end at $x$, then $[F(\overline{\alpha_1}) \cdots F(\overline{\alpha_r})]: Fx_1 \oplus \cdots \oplus Fx_r \rightarrow Fx$ is a sink morphism (i.e, a right minimal almost split morphism).
\end{enumerate}
\end{defi}

\begin{obs}
    To be precise,  Riedtmann \cite{Rie} and Bongartz-Gabriel \cite{BG} only define well-behaved functors when the covering $\pi$ is the universal covering $\tilde{\Gamma} \rightarrow \Gamma$, and Chaio-Le Meur-Trepode \cite{CMT1} only define them when $\pi$ is the generic covering $\hat{\Gamma} \rightarrow \Gamma$. But this assumption on $\pi$ is irrelevant for the above definition, and there are Riedtmann functors over other coverings. Supposing that $\pi$ is the universal/generic covering will actually play a role in the existence of Riedtmann functors, as we shall see below. Already in \cite{CMT2}, the authors define well-behaved functors over any covering.
\end{obs}

\begin{obs}
For a weakly Riedtmann functor $F:k(\Delta) \rightarrow \ind \Gamma$, if $\alpha$ is an arrow of $\Delta$, then $F(\overline{\alpha})$ is an irreducible morphism. (Actually this was the original condition used to define these functors in \cite{BG,Rie}, which only deal with algebras of finite type).

Reciprocally, only in the case where $\Gamma$ has trivial valuation, if $F:k(\Delta) \rightarrow \ind \Gamma$ is a $k$-linear functor such that $Fx = \pi x$ for every vertex $x$ of $\Delta$ and $F(\overline{\alpha})$ is an irreducible morphism for every arrow $\alpha$ of $\Delta$, then $F$ is a weakly Riedtmann functor.
\end{obs}

\subsection{Connecting both definitions of Riedtmann functors}

Having given two definitions of Riedtmann functors, we need to connect them with each other, thus proving that they are convergent, as we did in~\ref{subsec:mesh connection} with mesh categories. This will be done through our second main theorem stated right below, which also explains why the first definition has been called here `strong', and the other, `weak':

\begin{teo}
\label{th:universal x generico 2}

    Suppose $k$ is algebraically closed. Let $\Gamma$ be a component of the AR-quiver of an algebra $A$, and assume $\Gamma^v$ has the standard $k$-modulation. Let $\pi: \Delta \rightarrow \Gamma$ be a covering of $\Gamma$, with $\Delta^v$ having the modulation induced from the one of $\Gamma^v$.

    \begin{enumerate}

        \item If $F:k_a(\Delta) \rightarrow \ind \Gamma$ is a weakly Riedtmann functor, then there is a $k$-linear isomorphism of categories $\iota:k_g(\Delta^v) \xrightarrow{\sim} k_a(\Delta)$ such that $F \circ \iota: k_g(\Delta^v) \rightarrow \ind \Gamma$ is a strongly Riedtmann functor.

        \item If $i:k_a(\Delta) \xrightarrow{\sim} k_g(\Delta^v)$ is a $k$-linear isomorphism of categories and $F:k_g(\Delta^v) \rightarrow \ind \Gamma$ is a strongly Riedtmann functor, then $F \circ i:k_a(\Delta) \rightarrow \ind \Gamma$ is a weakly Riedtmann functor.
    \end{enumerate}
\end{teo}

The difference between `strong' and 'weak' Riedtmann functors is therefore justified by the logic quantifiers used for stating the existence of the isomorphisms above: a strongly Riedtmann functor becomes a weak one by composing with any isomorphism, whereas a weakly Riedtmann functor only becomes strong when composed with a specific isomorphism.

\begin{proof}
    \begin{enumerate}

    \item The strategy is similar to the one used to prove item 2 of Theorem~\ref{th:universal x generico} above: we prove the existence of the isomorphism $\iota$ by proving that the path categories $k_g \Delta^v$ and $k_a \Delta$ are isomorphic and then observe that this isomorphism preserves mesh relations.

    Consider a pair of vertices $x,y \in \Delta_0$ linked by at least one arrow in $\Delta$, and suppose $\alpha_1,\cdots,\alpha_d: x \rightarrow y$ are all the arrows between $x$ and $y$. Since $F$ is a weakly Riedtmann functor, $\{\overline{F(\overline{\alpha_1})},\dots,\overline{F(\overline{\alpha_d})}\}$ forms a basis of $\operatorname{irr}(\pi x, \pi y) \doteq M(x,y)$. That determines an isomorphism of $k$-vector spaces $\iota: (k_g \Delta)_1(x,y) = M(x,y) \rightarrow (k_a \Delta)_1(x,y)$, given by: if $h: \pi x \rightarrow \pi y$ is irreducible and such that $h = \sum_{i=1}^d \lambda_i F(\overline{\alpha_i})$ modulo $\rad^2$, then $i(\overline{h}) = \sum_{i=1}^d \lambda_i \overline{\alpha_i}$. Note from this definition that for $1 \leq i \leq d$, $\iota(\overline{F(\overline{\alpha_i})}) = \overline{\alpha_i}$.

    Given that $k_g \Delta^v$ and $k_a \Delta$ are tensor categories, this shows how to build a $k$-linear isomorphism $k_g\Delta^v \cong k_a \Delta$. Let us see that it preserves the relations.

    Let $y \in \Delta_0$ be a non-projective vertex, and let $y_1,\cdots, y_r$ be the immediate predecessors of $y$. For $i$ between 1 and $r$, let $\alpha_{i1}, \dots, \alpha_{id_i}: y_i \rightarrow y$ be the arrows between $y_i$ and $y$. Then $(\overline{F(\overline{\alpha_{i1}})},\dots, \overline{F(\overline{\alpha_{id_i}})})$ is a basis of $\operatorname{irr}(\pi y_i, \pi y)$. The mesh that ends in $y$ has the form:

    
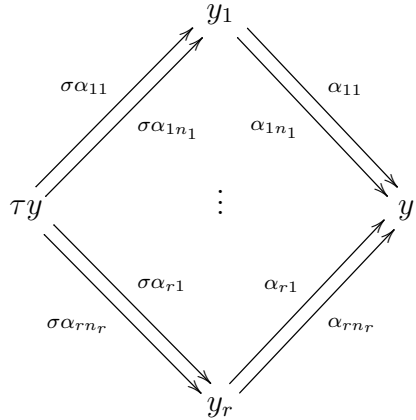
\begin{displaymath}
        \xymatrix{ && y_1 \ar@<0.5ex>[ddrr]^{\alpha_{11}} \ar@<-0.5ex>[ddrr]_{\alpha_{1n_1}} && \\
        &&&& \\
        \tau y \ar@<0.5ex>[uurr]^{\sigma \alpha_{11}} \ar@<-0.5ex>[uurr]_{\sigma \alpha_{1n_1}} \ar@<0.5ex>[ddrr]^{\sigma \alpha_{r1}} \ar@<-0.5ex>[ddrr]_{\sigma \alpha_{rn_r}} && \vdots && y\\
        &&&& \\
        && y_r \ar@<0.5ex>[uurr]^{\alpha_{r1}} \ar@<-0.5ex>[uurr]_{\alpha_{rn_r}} && }
    \end{displaymath}

    \textit{Fact.} For every $1 \leq i \leq r$, $(\overline{F(\overline{\sigma \alpha_{i1}})},\dots, \overline{F(\overline{\sigma \alpha_{id_i}})})$ (which is a subset of $ \operatorname{irr}(\pi \tau y, \pi y_i)$)  is the dual basis of $(\overline{F(\overline{\alpha_{i1}})},\dots, \overline{F(\overline{\alpha_{id_i}})}) \subseteq \operatorname{irr}(\pi y_i, \pi y)$ relative to the $k$-bilinear form $\sigma_y: \operatorname{irr}(\pi y_i, \pi y) \otimes \operatorname{irr}(\pi \tau y, \pi y_i) \rightarrow k$.

    Note that, if the fact above is true, then we have

    \begin{align*}
        \iota \left( \sum_{i=1}^r \sum_{j=1}^{d_i} (\overline{F(\overline{\alpha_{ij}})})(\overline{F(\overline{\alpha_{ij}})})^*\right) &= \iota \left( \sum_{i=1}^r \sum_{j=1}^{d_i} \overline{F(\overline{\alpha_{ij}})}\overline{F(\overline{\sigma \alpha_{ij}})}\right) \\
        &= \sum_{i=1}^r \sum_{j=1}^{d_i} \iota(\overline{F(\overline{\alpha_{ij}})})\iota(\overline{F(\overline{\sigma \alpha_{ij}})}) = \sum_{i=1}^r \sum_{j=1}^{d_i}  \overline{\alpha_{ij}}\overline{\sigma\alpha_{ij}} = 0
    \end{align*}

    Thus $\iota$ vanishes on the mesh relations of $k_g(\Delta^v)$, proving that $\iota$ induces a $k$-linear isomorphism $\iota: k_g(\Delta^v) \xrightarrow{\sim} k_a(\Delta)$.

    \textit{Proof of the fact.}

    Since $F: k_A(\Delta) \rightarrow \ind \Gamma$ is a weakly Riedtmann functor, we have that

    $$0 \rightarrow \pi \tau y \xrightarrow{(F(\overline{\sigma \alpha_{lm}}))_{l,m}} \bigoplus_{l=1}^r \bigoplus_{m = 1}^{d_l} \pi y_l \xrightarrow{(F(\overline{\alpha_{lm}}))_{l,m}} \pi y \rightarrow 0$$

    is an almost split sequence. Fix $1 \leq i \leq r$. If $1 \leq j,j' \leq d_i$, then we get that the following diagram commutes:

    \begin{displaymath}
        \xymatrix{&&& \pi y_i \ar[d]^{F(\overline{\alpha_{ij}})} \ar[dl]_{u}& \\
        0 \ar[r]& \pi \tau y \ar[d]_{F(\overline{\sigma \alpha_{ij'}})} \ar[r]& \bigoplus_{l=1}^r \bigoplus_{m = 1}^{d_l} \pi y_l \ar[dl]^p \ar[r]& \pi y \ar[r]& 0\\
        & \pi y_i &&&}
    \end{displaymath}

    where $u: \pi y_i \rightarrow \bigoplus_{l=1}^r \bigoplus_{m = 1}^{d_l} \pi y_l$ is the canonical inclusion of the $j$-th copy of $\pi y_i$ in the direct sum $\bigoplus_{l=1}^r \bigoplus_{m = 1}^{d_l} \pi y_l$, and $p: \bigoplus_{l=1}^r \bigoplus_{m = 1}^{d_l} \pi y_l \rightarrow \pi y_i$ is the projection of the $j'$-th copy. 

    By definition of $\sigma_y$, we have $\sigma_y(\overline{F(\overline{\alpha_{ij}})} \otimes \overline{F(\overline{\sigma \alpha_{ij'}})} = \overline{pu} = \delta_{jj'} \in k \cong \operatorname{End} (\pi y_i) / \rad \operatorname{End} (\pi y_i)$, where $\delta_{jj'}$ is Kronecker's delta. That concludes the proof of the fact.

    Having proved the fact, it remains to prove that the functor $F \circ \iota: k_g(\Delta^v) \rightarrow \ind \Gamma$ is strongly Riedtmann. Let $h: \pi x \rightarrow \pi y$ be an irreducible morphism, where $x,y$ are vertices of $\Delta$. We want to show that $h - F\circ \iota (\overline{h}) \in \rad^2$.

    Let $\alpha_1,\dots,\alpha_d:x \rightarrow y$ be the arrows between $x$ and $y$. Then there are $\lambda_1,\dots,\lambda_d \in k$ such that $\iota(\overline{h}) = \lambda_1 \overline{\alpha_1} + \cdots + \lambda_d \overline{\alpha_d}$. We have

    \begin{align*}
        \iota(\overline{h}) &= \lambda_1 \overline{\alpha_1} + \cdots + \lambda_d \overline{\alpha_d}\\
        &= \lambda_1 \iota(\overline{F(\overline{\alpha_1})}) + \cdots + \lambda_d \iota(\overline{F(\overline{\alpha_d})}) = \iota \left(\lambda_1\overline{F(\overline{\alpha_1})}) + \cdots + \lambda_d\overline{F(\overline{\alpha_d})}\right)
    \end{align*}

    Using that $\iota$ is an isomorphism, this tells us that $\overline{h} = \lambda_1\overline{F(\overline{\alpha_1})} + \cdots + \lambda_d\overline{F(\overline{\alpha_d})} \break = \overline{F(\lambda_1\overline{\alpha_1} + \cdots + \lambda_d\overline{\alpha_d})} = \overline{F(\iota(\overline{h}))}$, that is, $h - F\circ \iota (\overline{h}) \in \rad^2$, concluding the proof of this item. 
    
    \item To show that $F \circ i$ meets the definition of weakly Riedtmann functor, we use basic Auslander-Reiten theory (for example, we could invoke Proposition IV.1.2 from \cite{AC}), and so it suffices to show the following fact:

    \textit{Fact.} If $\alpha_1,\dots,\alpha_d$ are all the arrows between two given vertices $x,y \in \Delta_0$, then the set $\{F \circ i(\overline{\alpha_1}) + \rad^2, \dots, F \circ i(\overline{\alpha_d})+ \rad^2\}$ forms a $k$-basis of $\operatorname{irr} (\pi x, \pi y)$.

    Let us prove the fact above. Since there are $d$ arrows $x \rightarrow y$ in $\Delta$ and $\pi: \Delta \rightarrow \Gamma$ is a covering, there is a total of $d$ arrows $\pi x \rightarrow \pi y$ in $\Gamma$. Thus $d = \dim_k \operatorname{irr} (\pi x, \pi y)$. Therefore to show that $\{F \circ i(\overline{\alpha_1}) + \rad^2, \cdots, F \circ i(\overline{\alpha_d})+ \rad^2\}$ is a $k$-basis of $\operatorname{irr} (\pi x, \pi y)$, it suffices to show that it is $k$-linearly independent.

    Let then $\lambda_1,\dots,\lambda_d \in k$ be scalars such that $\lambda_1 F \circ i(\overline{\alpha_1}) + \cdots +\lambda_d F \circ i(\overline{\alpha_d}) \in \rad^2(\pi x, \pi y)$. We have that, if $h \in \rad_A(\pi x, \pi y)$ is such that $\overline{h} = i(\lambda_1 \overline{\alpha_1} + \cdots + \lambda_d \overline{\alpha_d}) \in \operatorname{irr} (\pi x, \pi y)$, then $F(\overline{h}) = F \circ i (\lambda_1 \overline{\alpha_1} + \cdots + \lambda_d \overline{\alpha_d}) = \lambda_1 F \circ i(\overline{\alpha_1}) + \cdots \lambda_d F \circ i(\overline{\alpha_d}) \in \rad^2$.

    Since $F$ is a strongly Riedtmann functor, we know that $h - F(\overline{h}) \in \rad^2$, and since $F(\overline{h}) \in \rad^2$ it follows that $h \in \rad^2(\pi x, \pi y)$, that is, $\overline{h} = 0$. By the definition of $h$, it holds that $i(\lambda_1 \overline{\alpha_1} + \cdots + \lambda_d \overline{\alpha_d}) = 0$. Since $i$ is an isomorphism, we get $\lambda_1 \overline{\alpha_1} + \cdots + \lambda_d \overline{\alpha_d} = 0$. But the elements $\overline{\alpha_1}, \dots, \overline{\alpha_d}$ freely generate the morphisms of degree 1 between $x$ and $y$, and from that we may conclude that $\lambda_1 = \cdots = \lambda_d = 0$, which ends the proof.
    \end{enumerate}
\end{proof}

\subsection{The existence of Riedtmann functors}

So far we have only dealt with the definition of Riedtmann functors, and now it is time to talk about their existence. This has been asserted a couple of times in the literature, with a series of results that get increasingly strong.

Generally speaking, we cannot expect to have a Riedtmann functor $k(\Gamma) \rightarrow \ind \Gamma$ for every component $\Gamma$ of an AR-quiver, since oriented cycles in $\Gamma$ usually make this existence more unlikely. We will address components for which Riedtmann functors $k(\Gamma) \rightarrow \ind \Gamma$ exist in a forthcoming paper (\cite{CC6}).

However, once we consider a covering $\pi:\Delta \rightarrow \Gamma$, the existence of a Riedtmann functor $k(\Delta) \rightarrow \ind \Gamma$ becomes more likely, and if $\Delta$ satisfies additional properties, then this existence can be assured. So the basic result about existence of Riedtmann functors is the following:

\begin{teo}
\label{th:existe funtor de Ri}
    Let $\Gamma$ be a component of the AR-quiver of an algebra $A$. Then there is a covering $\pi:\Delta \rightarrow \Gamma$ for which there is a Riedtmann functor $F: k(\Delta) \rightarrow \ind \Gamma$.
\end{teo}

The previous theorem has been proved in the literature under the following conditions:

\begin{itemize}
    \item It was proved in the case where $k$ algebraically closed, $\Gamma$ is finite and stable, $A$ is self-injective, and with $\Delta$ having the form $\mathbb{Z}B$, where $B$ is a tree (Riedtmann, '80).

    \item It was proved for $k$ algebraically closed, $\Gamma$ being finite, and with $\Delta$ being the \textit{universal covering} of $\Ga$ (Bongartz, Gabriel, '82).

    \item Then it was proved for every $\Ga$ (not necessarily finite), with $k$ algebraically closed, but $\Delta$ being the \textit{generic covering} (Chaio, Le Meur, Trepode, '11). In this case and in the ones above, only the existence of weakly Riedtmann functors is dealt with.

    \item Finally, it was proved for every $\Ga$, $k$ being \textit{perfect}, and $\Delta$ being the \textit{universal covering} (Chaio, Le Meur, Trepode, '19). It is in this article that what we call strongly Riedtmann functors are defined and have their existence shown.
\end{itemize}

\subsection{A key property of Riedtmann functors}
\label{sec:th b}

We now state a key property of Riedtmann functors, that relates the filtrations of the radical of the mesh category with the ones of the module category.

Originally, this property was stated in \cite{BrG,CT} for standard components (i.e., components for which there is an isomorphism $k(\Gamma) \cong \ind \Gamma$), and then extended for any weakly Riedtmann functor in \cite{CMT1}. Eventually it was proved for strongly Riedtmann functors in \cite{CMT2}:

\begin{teo}[\cite{CMT2}, Thm. B]
\label{th:b19}
Suppose $k$ is perfect. Let $\Gamma$ be a component of the AR-quiver of a $k$-algebra $A$, and let $\pi: \Delta \rightarrow \Gamma$ be a covering, for which there is a strongly Riedtmann functor $F:k(\Delta) \rightarrow \ind \Gamma$, where $\Delta$ is equipped with the modulation induced from the standard modulation in $\Gamma$. Then for every $n \geq 1$ and every pair of vertices $x,y \in \Delta_0$, the functor $F$ induces bijections

$$\bigoplus_{Fz = Fy} \frac{\mathcal{R}^nk(\Delta)(x,z)}{\mathcal{R}^{n+1}k(\Delta)(x,z)}\xrightarrow{\sim} \frac{\rad^n(Fx,Fy)}{\rad^{n+1}(Fx,Fy)}$$

and

$$\bigoplus_{Fz = Fx} \frac{\mathcal{R}^nk(\Delta)(z,y)}{\mathcal{R}^{n+1}k(\Delta)(z,y)}\xrightarrow{\sim} \frac{\rad^n(Fx,Fy)}{\rad^{n+1}(Fx,Fy)}$$

\end{teo}

\section{Applications to compositions of irreducible morphisms}
\label{sec:appl to composites}

\subsection{A criterion for general composites of irreducible morphisms}
\label{subsec:inf ou atalhos}

Having defined mesh categories and Riedtmann functors in the previous sections, we now want to state our main application here, which uses these concepts to the problem of composites of irreducible morphisms.

First we introduce some terminology that will be useful in the statement of the result.

\begin{defi} 
    \begin{enumerate}
        \item Let  $\Gamma$ be a component  of an AR-quiver. We say that a path $X_0 \xrightarrow{\alpha_1} X_1 \xrightarrow{\alpha_2} \cdots \xrightarrow{\alpha_n} X_n$  in $\Gamma$ is a \textbf{shortcut} for a path  $X_0 = Y_0 \xrightarrow{\beta_1} Y_1 \xrightarrow{\beta_2} \cdots \xrightarrow{\beta_m} Y_m = X_n$ in $\Gamma$ if $n<m$. 
        
        \item We say that a composite of irreducible morphisms $X_0 \xrightarrow{f_1} X_1 \xrightarrow{f_2} \cdots \xrightarrow{f_n} X_n$ between indecomposable modules is a \textbf{shortcut} for another composite of irreducible morphisms $X_0 = Y_0 \xrightarrow{g_1} Y_1 \xrightarrow{g_2} \cdots \xrightarrow{g_m} Y_m = X_n$ between indecomposable modules if $n<m$.
    \end{enumerate}
\end{defi}

\begin{obs}
    If $\alpha: i \rightarrow j$ is an arrow of $\Ga(\md A)$, then $\alpha$ is a shortcut of a path if and only if $\alpha$ has a \textit{bypass} (see \cite{CHR} or \cite{CT} for the definition) or there is an oriented cycle passing through $i$ or $j$.
\end{obs}

We are ready for the statement of the result. It is a characterization of when the non-zero composition of $n$ irreducible morphisms belongs to the $n+1$-power of the radical.

\begin{teo}
\label{th:inf ou atalhos}

Let $A$ be a finite dimensional algebra over a perfect field $k$ and $\Ga$ be a component of $\Ga(\md A)$. 
Given indecomposable modules $X_0,X_1,\cdots,X_n$ in $\Ga$, the following are equivalent:

\begin{enumerate}
    \item There is a path of irreducible morphisms $X_0 \xrightarrow{h_1} X_1 \xrightarrow{h_2} \cdots \xrightarrow{h_n} X_n$ such that $h_n \cdots h_1$ is non-zero and belongs to $\rad^{n+1}(X_0,X_n)$.

    \item There is a path of irreducible morphisms $X_0 \xrightarrow{f_1} X_1 \xrightarrow{f_2} \cdots \xrightarrow{f_n} X_n$ with $f_n \cdots f_1 =0$ and there are morphisms $X_0 \xrightarrow{\epsilon_1} X_1 \xrightarrow{\epsilon_2} \cdots \xrightarrow{\epsilon_n} X_n$ such that $\epsilon_n \cdots \epsilon_1 \neq 0$ and satisfying that, for every $1 \leq i \leq n$, either $\epsilon_i \in \rad^2(X_{i-1},X_i)$ or $\epsilon_i = f_i$.

    \item There is a path of irreducible morphisms $X_0 \xrightarrow{f_1} X_1 \xrightarrow{f_2} \cdots \xrightarrow{f_n} X_n$ with $f_n \cdots f_1 =0$, and such that one of the following holds:
    \begin{enumerate}
        
        \item  There is a path of irreducible morphisms $X_0 \xrightarrow{h_1} X_1 \xrightarrow{h_2} \cdots \xrightarrow{h_n} X_n$ such that $h_n \cdots h_1 \in \rad^{\infty}(X_0,X_n)\setminus\{0\}$.
    
        \item There are indices $i_1 < \cdots < i_l$ such that for every $1 \leq j \leq l$, $f_{i_j}$ is a shortcut for some composition of irreducible morphisms $\phi_{i_j}$, and such that

    $$f_n \cdots f_{i_l+1} \phi_{i_l} f_{i_l-1} \cdots f_{i_1+1} \phi_{i_1} f_{i_1-1} \cdots f_1 \neq 0.$$
    \end{enumerate}
\end{enumerate}
\end{teo}

\begin{proof}

First of all, using Theorem~\ref{th:existe funtor de Ri}, fix a strongly Riedtmann functor $F: k(\widetilde{\Ga^v}) \rightarrow \ind \Gamma$, where $\pi: \widetilde{\Ga^v} \rightarrow \Gamma^v$ is the universal covering.

$(1) \Leftrightarrow (2)$: is done in \cite{CMT2}, Proposition 3.

    $(1) \Rightarrow (3)$: Fix some $x_0 \in \pi^{-1}(X_0)$. Then, we can lift the path from the hypothesis to a path $\gamma: x_0 \rightarrow x_1 \rightarrow \cdots \rightarrow x_n$ over $\tilde{\Ga}$, whose image under $\pi$ is the given path $X_0 \rightarrow X_1 \rightarrow \cdots \rightarrow X_n$ over $\Ga$. Since each $h_i: X_{i-1} \rightarrow X_i$ is irreducible, we have that $\overline{h_i} \in \rad(X_{i-1},X_i)/\rad^2(X_{i-1},X_i)$, and also that $h_i - F(\overline{h_i}) \in \rad^2(X_{i-1},X_i)$, using that $F$ is a Riedtmann functor.

    Applying Theorem~\ref{th:b19} successively, we realize that for every $i$ there is an element $(\phi_z)_z \in \oplus_{Fz = Fx_i} \mathcal{R}^2 k(\tilde{\Ga})(x_i,z)$ such that $h_i - F(\overline{h_i}) - \sum_z F(\phi_z) \in \rad^{\infty}(X_{i-1},X_i)$. Then define $h_{i1} = F(\overline{h_i})$, $h_{i2} = \sum_z F(\phi_z)$ e $h_{i3} = h_i - h_{i1} - h_{i2} \in \rad^{\infty}$.

    Observe that

    \begin{align*}
        h_n \cdots h_1 &= (h_{n1} + h_{n2} + h_{n3})  \cdots (h_{11} + h_{12} + h_{13}) \\
        &= \sum_{j_1,\cdots,j_n = 1}^3 h_{nj_n} \cdots h_{2j_2}h_{1j_1}
    \end{align*}

    Note that, for those terms having any $j_i \geq 2$, it holds that $h_{ij_i} \in \rad^2$ and therefore $h_{nj_n} \cdots h_{1j_1} \in \rad^{n+1}(X_0,X_n)$. Thus, since by hypothesis $h_n \cdots h_1 \in \rad^{n+1}(X_0,X_n)$, we must have $h_{n1} \cdots h_{11} \in \rad^{n+1}(X_0,X_n)$. That is, $F(\overline{h_n}) \cdots F(\overline{h_1}) = F(\overline{h_n} \cdots \overline{h_1}) \in \rad^{n+1}(X_0,X_n)$. By Theorem~\ref{th:b19}, it follows that $\overline{h_n} \cdots \overline{h_1} \in \mathcal{R}^{n+1}k(\tilde{\Ga})(x_0,x_n)$, and using that $k(\tGa)$ is $\mathbb{N}$-graded, it follows that $\overline{h_n} \cdots \overline{h_1} = 0$. Thus $h_{n1} \cdots h_{11} = 0$ and it is sufficient to take $f_i = h_{i1}$ to conclude the first part of (3).

    For the second part, we have two cases:

    \begin{enumerate}
        \item[a)] Suppose that, for every choice of $j_1, \cdots, j_n$, if $h_{nj_n} \cdots h_{1j_1} \neq 0$ then there is an $i$ with $j_i = 3$.

        Fix $j_1, \cdots, j_n$ com $h_{nj_n} \cdots h_{1j_1} \neq 0$. Then there is an $i$ with $j_i = 3$ and since $h_{ij_i} = h_{i3} \in \rad^{\infty}$, we have $h_{nj_n} \cdots h_{1j_1} \in \rad^{\infty}$. So the hypothesis we have for this case implies that 

        $$h_n \cdots h_1 = \sum_{j_1,\cdots,j_n = 1}^3 h_{nj_n} \cdots h_{2j_2}h_{1j_1} \in \rad^{\infty}(X_0,X_n)$$
        
        \item[b)] Suppose that we have the opposite of a), that is, there are $j_1, \cdots, j_n$ such that $j_i < 3$ for every $i$ and with $h_{nj_n} \cdots h_{1j_1} \neq 0$.

        Denote $\{i_1,\cdots,i_l\} = \{i: j_i = 2\}$, where $i_1 < \cdots < i_l$. If $i$ is such that $j_i=1$, then $h_{ij_i} = f_i$, and therefore we have

        $$f_n \cdots f_{i_l+1} h_{i_l 2} f_{i_l-1} \cdots f_{i_1+1} h_{i_1 2} f_{i_1-1} \cdots f_1 \neq 0$$

        For every $i$ such that $j_i =2$, $h_{ij_i}$ is, by construction, a sum of compositions of at least two irreducible morphisms. That way, we can choose for each such $i$ one of these compositions, which we denote by $\phi_i$, in such a way that we have

        $$f_n \cdots f_{i_l+1} \phi_{i_l} f_{i_l-1} \cdots f_{i_1+1} \phi_{i_1} f_{i_1-1} \cdots f_1 \neq 0$$

        By construction, each $\phi_i$ has the morphism $f_i$ as a shortcut. This concludes the proof of $(1) \Rightarrow (3)$.
    \end{enumerate}

    $(3) \Rightarrow (1)$: Let $f_1, \cdots, f_n$ be as stated. If (a) holds, then it obvious that (1) also holds. So suppose that (b) holds. We shall construct $h_1,\cdots,h_n$ as in (1).

    Given a subsequence $j_1<\cdots< j_m$ of the sequence of indices $i_1 < \cdots < i_l$, we can consider the morphism

    $$g(j_1,\cdots,j_m) \doteq f_n \cdots f_{j_m+1} \phi_{j_m} f_{j_m-1} \cdots f_{j_1+1} \phi_{j_1} f_{j_1-1} \cdots f_1$$

    Let $m_0$ be the natural number which is the minimum length of a subsequence $j_1<\cdots< j_{m_0}$ for which $g(j_1,\cdots, j_{m_0}) \neq 0$. Since the hypothesis (b) says that $g(i_1,\cdots,i_l) \neq 0$, there must be one such $m_0$, and since $g(\emptyset) = f_n \cdots f_1 = 0$, we have that $m_0 > 0$.

    Now define

    $$h_i = \begin{cases}
        f_i & \text{if } i \notin \{j_1,\cdots,j_{m_0}\}\\
        f_i+\phi_i & \text{if } i \in \{j_1,\cdots,j_{m_0}\}
    \end{cases}$$

    Since for each $i$, $f_i$ is irreducible and $\phi_i \in \rad^2(X_{i-1},X_i)$, we have that $h_i$ is irreducible. Let us see that $h_n \cdots h_1 \in \rad^{n+1}(X_0,X_n) \setminus \{0\}$. Observe that

    \begin{align*}
        h_n \cdots h_1 &=  f_n \cdots f_{j_{m_0}+1} (f_{m_0}+\phi_{j_{m_0}}) f_{j_{m_0}-1} \cdots f_{j_1+1} (f_{j_1}+\phi_{j_1}) f_{j_1-1} \cdots f_1\\
        &=g(j_1,\cdots, j_{m_0}) + \sum_{\substack{\{p_1 < \cdots < p_t\} \subseteq \{j_1 < \cdots < j_{m_0}\}\\ t<m_0}} g(p_1,\cdots,p_t)
    \end{align*}

    By the minimality of $m_0$, each term $g(p_1,\cdots,p_t)$ from the summation above is zero, and so we have $h_n \cdots h_1 = g(j_1,\cdots, j_{m_0}) \neq 0$. Moreover, we have 
    
    $$g(j_1,\cdots, j_{m_0}) = f_n \cdots f_{j_{m_0}+1} \phi_{j_{m_0}} f_{j_{m_0}-1} \cdots f_{j_1+1} \phi_{j_1} f_{j_1-1} \cdots f_1$$ 
    
    with $m_0 > 1$ and $\phi_{j} \in \rad^2(X_{j-1},X_j)$ for every $j$. Therefore $h_n \cdots h_1 = g(j_1,\cdots, j_{m_0}) \in \rad^{n+1}(X_0,X_n)$.
\end{proof}

\begin{obs}
    It is important to detail our references for the theorem above. \cite{CT}, Theorem 2.7 proves $(2) \Rightarrow (1)$, and proves $(1) \Rightarrow (2)$ in the particular case where $k$ is algebraically closed and $\Gamma$ is a standard component with trivial valuation. \cite{CMT1}, Proposition 5.1 drops the assumption that $\Gamma$ is standard for $(1) \Rightarrow (2)$. Eventually, \cite{CMT2}, Proposition 3 states $(1) \Leftrightarrow (2)$ with the same generality as here. Parallel to that, \cite{CT} proves $(1) \Rightarrow (3)$ through its Lemma 2.8 and Corollary 2.9, although their statement is weaker than what we do here. We did not find the implication $(3) \Rightarrow (1)$ in literature, but we have written the proof of it above inspired by the argument in the second part of the proof of \cite{CT}, Theorem 2.7. In summary, Theorem~\ref{th:inf ou atalhos} above gives a unified statement of some results or arguments by C. Chaio, P. Le Meur and S. Trepode, which were originally stated in multiple sources or not explicitly stated.
\end{obs}

Observe that it follows from the argument we gave as we proved $(1) \Rightarrow (2)$ of Theorem~\ref{th:inf ou atalhos} the following corollary:

\begin{cor}[from the proof of Theorem~\ref{th:inf ou atalhos}]
\label{cor:inf ou atalhos}

Let $A$ be a finite dimensional algebra over a perfect field $k$ and $\Ga$ be a component of $\Ga(\md A)$. Let $F:k(\Delta) \rightarrow \ind \Gamma$ be a Riedtmann functor, where $\pi: \Delta \rightarrow \Gamma^v$ is a covering. Given a path of irreducible morphisms $X_0 \xrightarrow{h_1} X_1 \xrightarrow{h_2} \cdots \xrightarrow{h_n} X_n$ between modules from $\Ga$, let $x_0 \rightarrow x_1 \rightarrow \cdots \rightarrow x_n$ be a path over $\Delta$ such that $\pi(x_i) = X_i$ for every $0 \leq i \leq n$. Consider $\overline{h_i}$ as an element of $k(\Delta)_1(x_{i-1},x_i)$ for each $i$. With these notations, the following are equivalent:

\begin{enumerate}
    \item $h_n \cdots h_1 \in \rad^{n+1}$;
    \item $F(\overline{h_n}) \cdots F(\overline{h_1}) \in \rad^{n+1}$;
    \item $\overline{h_n} \cdots \overline{h_1} \in \mathcal{R}^{n+1} k(\Delta)$;
    \item $\overline{h_n} \cdots \overline{h_1} = 0$ in $k(\Delta)$; and
    \item $F(\overline{h_n}) \cdots F(\overline{h_1}) = 0$.
\end{enumerate}

\end{cor}

\subsection{A new proof for the Igusa-Todorov Theorem}
\label{subsec:new proof IT}

As we stated in the introduction, the Igusa-Todorov Theorem is a classical result in compositions of irreducible morphisms. The proof originally given in \cite{IT1} (which can also be found in \cite{AC}, IV.3.5) is elementary, but our last goal here is to show that we can use Riedtmann functors to give a quite more elegant proof in the case where $k$ is perfect:

\begin{proof}[Proof of Theorem~\ref{th:igusa todorov}]
    Consider the universal covering $\pi: \tGa \rightarrow \Ga$ of the component which contains the morphisms $h_1, \cdots, h_n$. Fix a vertex $x_0 \in \pi^{-1}(X_0)$ and consider a lifting of the path $h_n \cdots h_1$ to $\tGa$. If we have $h_n \cdots h_1 \in \rad^{n+1}(X_0,X_n)$, then Corollary~\ref{cor:inf ou atalhos} tells us that $\overline{h_n} \cdots \overline{h_1} = 0$. On the other side, since mesh relations generate linear combinations of paths which are necessarily not sectional, there is no way a sectional path can be identified with zero via mesh relations. That means $\overline{h_n} \cdots \overline{h_1} \neq 0$, a contradiction that concludes the proof.
\end{proof}

\section*{Acknowledgements}

This work is part of the PhD thesis (\cite{C-th}) of first named author, under supervision by the second named author. The authors gratefully acknowledge financial support by S\~ao Paulo Research Foundation - FAPESP (grants \#2020/13925-6 and \#2022/02403-4), and by CNPq (grant Pq 312590/2020-2).

\end{document}